%% file: PAPER.tex
\documentclass[final, 11pt]{article}

\usepackage{amsmath}
\usepackage{amsfonts}
\usepackage{makeidx}
\usepackage{amsthm}
\usepackage{mathtools}
\usepackage{amssymb}
\usepackage{bm}
\usepackage{bbm}
\usepackage{cite}
\usepackage{showkeys}
\usepackage[english]{babel}
\usepackage{dblaccnt}
\usepackage{accents}
\usepackage{graphicx}
\usepackage{psfrag}
\usepackage{subfig}
\usepackage{color}
\usepackage{float}
\usepackage{amscd}
\usepackage[mathscr]{euscript}
\usepackage{lipsum}
\usepackage{epstopdf}
\usepackage{algorithm}
\usepackage{algpseudocode}

\input{mathmacros.tex}

\newtheorem{defi}{Definition}
\newtheorem{lemma}[defi]{Lemma}

\begin{document}

\title{On  reconstruction of small  sources from Cauchy data at a fixed frequency}
\author{Isaac Harris\thanks{Department of Mathematics, Purdue University, West Lafayette, IN 47907, USA; (\texttt{harri814@purdue.edu})} \and Thu Le\thanks{Department of Mathematics, Kansas State University, Manhattan, KS 66506, USA; (\texttt{thule@ksu.edu})} \and Dinh-Liem Nguyen\thanks{
Department of Mathematics, Kansas State University, Manhattan, KS 66506, USA; (\texttt{dlnguyen@ksu.edu}).
 }}

\date{}
\maketitle

\begin{abstract}
This short paper is concerned with the numerical reconstruction of  small  sources from boundary Cauchy data for a single  frequency. 
We study a sampling method to determine the location of small sources  in a very fast and robust way. Furthermore, the method can also 
  compute the intensity of point sources provided that the sources  are well separated. A simple justification of the method is done using  the Green representation formula 
and an asymptotic expansion of the radiated field for small volume sources. The implementation of the method  is non-iterative, 
computationally cheap, fast, and very simple. Numerical examples are presented to illustrate the performance of the method.

\end{abstract}

\sloppy


\bigskip


\section{Introduction}
\label{intro}

We study in this paper a numerical method for the inverse source problem of determining small radiating sources from Cauchy data  at a fixed frequency. 
The Cauchy data are  boundary measurements of the field generated by the source.  Inverse source problems  have
 applications in  physics and engineering such as antenna synthesis and medical imaging. These inverse problems  
 have been studied intensively during the past two decades. We refer to \cite{Acosta2012, ElBadia2013, Cheng2016, Bao2020, Li2021} and references therein for  results on stability analysis of inverse source problems.  
There has been a large amount of literature on numerical methods for solving these inverse  problems. 
Results on numerical methods  for inverse source problems with multi-frequency data can be found in~\cite{Eller2009, Bao2011, Bao2015, Zhang2015,  Griesmaier2017, LHNguyen2019, Alzaalig2020, Ji2021}.

In the case of data associated with a single frequency  inverse source problems may have multiple solutions  unless some a priori information about the source is given~\cite{Alves2009, ElBadia2011}. We refer to~\cite{ElBadia2000, ElBadia2005, Nara2008, Chung2009} for an algebraic method for computing point  sources  for  elliptic partial differential equations such as the Poisson equation and the Helmholtz equation.  An iterative approach of Newton type can be found in~\cite{Alves2009, Kress2013} for recovering  point  and extended sources.
 A direct sampling method for computing multipolar point sources for the Helmholtz equation   was studied in \cite{Zhang2019}. The main advantage of the sampling approach is that it  is  non-iterative,  fast, and computationally cheap. In the present work, we  also  study a sampling method  for recovering  point sources and small volume sources for the Helmholtz equation. Compared with the imaging functional  introduced in~\cite{Zhang2019} the  imaging functional  in this paper involves only a single integral instead of 
 a double integral. Thus the sampling  method  in the present work may be  faster and computationally cheaper (e.g. for  large search domains in three dimensions).  
 In addition, our implementation  does not require a multi-level sampling process as in~\cite{Zhang2019} and seems to provide equivalently good reconstruction results. 
 The justification of our sampling method for both two and three dimensions is done    simultaneously  in a pretty  simple way. 
It relies on   the Green representation formula and an asymptotic expansion of the radiated field for small volume sources.

The paper is organized as follows. Section 2 is dedicated to the reconstruction of the location and intensity of  point sources. The method is extended to locate sources with small volumes in Section 3. A numerical study of the performance of  the method is presented in Section 4. 

\section{Reconstruction of point sources}
We consider $N\in \mathbb{N}$  point sources located at $x_j$ for $j = 1,  2, \dots, N$ such that $\text{dist}(x_i , x_j) \geq c_0>0$. These sources are represented by the Delta distributions $\delta_{x_j}$ and have nonzero  intensities $\alpha_j \in \C$. Let $k>0$ be the wavenumber. Assume that the sources generate the radiating scattered field  $u$ that satisfies the following model problem
\begin{align}
\label{Helm}
& \Delta u + k^{2} u = - \sum_{j = 1}^N \alpha_j \delta_{x_j},\quad  \text{in } \mathbb{R%
}^{d},  \\
\label{radiation}
& \lim_{r\rightarrow \infty }r^{\frac{d-1}{2}}\left( \frac{\partial u}{\partial r} - \i ku\right) =0,\quad r=|x|,
\end{align}
where  $d$ = 2 or 3 is the dimension and the Sommerfeld radiation condition~\eqref{radiation} holds
uniformly for all  directions $ \hat{x} \in  \{\hat{x}\in \R^d: |\hat{x}| = 1\}$. The radiating scattered field  $u$ is given by 
\begin{align}
\label{usc}
u(x) = \sum_{j=1}^N \alpha_j  \,\Phi(x,x_j),
\end{align}
where  $\Phi(x,y)$ is the Green function of problem~\eqref{Helm}--\eqref{radiation} 
\begin{equation} 
\label{green}
 \Phi(x,y)= 
  \left\{ 
\begin{array}{cl}
\frac{\text{i}}{4}H^{(1)}_0(k|x-y|), & \text{in } \R^2 , \\ 
\frac{e^{\text{i}k|x-y|}}{4\pi|x-y|}, & \text{in } \R^3,
	\end{array}
	\right.
\end{equation}
and  $H_{0}^{(1)}$ is the first kind Hankel function of order zero. 
Let  $\Omega$ be a regular bounded   domain with boundary $\partial \Omega$ in $\R^d$. Assume that $x_j \in \Omega$ for $j = 1, \dots, N$ 
and denote by $\nu(x)$ the outward normal unit vector to $\partial \Omega$ at $x$.

\vspace{0.1cm}
\noindent
\textbf{Inverse problem:} Given the Cauchy data $u$ and $ \partial u/ \partial \nu$ on $\partial \Omega$ for a fixed wave number $k$, determine  locations $x_j$ and intensities $\alpha_j$ for $j = 1, 2, \dots, N$. 

We note that with a single wave number $k$ the Cauchy data is important to justify the uniqueness of solution to this inverse problem, see~\cite{Alves2009, ElBadia2011}.
The following simple lemma is important for the resolution analysis of the imaging functional we study for the inverse source problem. 
\begin{lemma}
The following identity holds for the radiating scattered field  $u$ in~\eqref{usc}
\begin{align}
\label{identity}
 \int_{\partial \Omega} \left(\frac{\partial \Im \Phi(x,z)}{\partial \nu(x)} u(x) - \Im \Phi(x,z) \frac{\partial u(x)}{\partial \nu(x)} \right) \dd s(x)= \sum_{j=1}^N \alpha_j  \,\Im \Phi(x_j,z).
\end{align} 
%$\Im \Phi(z,x_j)= \frac{1}{4}J_0(k|z-x_j|)$ in $\R^2$ and  $\Im \Phi(z,x_j) = \frac{k}{4\pi}j_0(k|z-x_j|)$  in $\R^3$
\end{lemma}
\begin{proof}
Substituting the formula for $u$ in \eqref{usc} in the left hand side of~\eqref{identity} implies
\begin{align}
& \int_{\partial \Omega} \left(\frac{\partial \Im \Phi(x,z)}{\partial \nu(x)}  \sum_{j=1}^N \alpha_j  \,\Phi(x,x_j) - \Im \Phi(x,z)  \sum_{j=1}^N \alpha_j  \frac{\partial \Phi(x,x_j)}{\partial \nu(x)} \right) \dd s(x) \nonumber \\
% & =  \int_{\partial\Omega} \left (\frac{\partial \Im \Phi(x,z)}{\partial \nu(x)} k^2\int_D  \Phi(x,y) \eta(y)u(y) \dd y - \Im \Phi(x,z) k^2  \int_\Omega \frac{\partial  \Phi(x,y)}{\partial \nu(x)}\eta(y)u(y) \dd y \right) \dd s(x) \nonumber \\
% & = k^2  \int_{\partial\Omega}  \int_D \left(\frac{\partial \Im \Phi(x,z)}{\partial \nu(x)} \Phi(x,y)- \Im \Phi(x,z) \frac{\partial  \Phi(x,y)}{\partial \nu(x)}\right) \, \eta(y) u(y,d) \dd y \,\dd s(x) \nonumber \\
\label{int1}
 &=  \sum_{j=1}^N \alpha_j  \int_{\partial \Omega}  \left(\frac{\partial \Im \Phi(x,z)}{\partial \nu(x)} \Phi(x_j,x) - \Im \Phi(x,z) \frac{\partial \Phi(x_j,x)}{\partial \nu(x)} \right)\dd s(x). 
\end{align}
%Recall $y \neq z$ since $y \in \Omega$ (measurement boundary) and $z$ is in the sampling domain $\Omega_S$. Hence,  
Due to the fact that $\Delta \Im \Phi(z,y) + k^2 \Im \Phi(z,y) = 0$ for  all $z, y \in \mathbb{R}^d$ the proof follows from the Green representation formula i.e. 
$$ \int_{\partial \Omega}  \left(\frac{\partial \Im \Phi(x,z)}{\partial \nu(x)} \Phi(y,x) - \Im \Phi(x,z) \frac{\partial \Phi(y,x)}{\partial \nu(x)} \right)\dd s(x) = \Im \Phi(y,z)$$
which proves the claim.
\end{proof}
Now, we define the imaging functional 
\begin{align}
\label{functional}
I(z):= \int_{\partial \Omega} \left(\frac{\partial \Im \Phi(x,z)}{\partial \nu(x)} u(x) - \Im \Phi(x,z) \frac{\partial u(x)}{\partial \nu(x)} \right) \dd s(x), 
\quad z \in \R^d.
\end{align}
Therefore, from the lemma above we have that 
$$ 
I(z) =  \sum_{j=1}^N \alpha_j  \,\Im \Phi(x_j,z), 
$$
where
 \begin{equation}
 \label{imag-bessel}
\Im \Phi(x,z) = \left\{ 
		\begin{array}{cl}
			\frac{1}{4} J_{0}(k|x - z|) &\text{in } \R^2, \\
			 \frac{k}{4\pi} j_0(k | x - z|)  &\text{in } \R^3,
		\end{array}
	\right.
\end{equation} 
and $J_0$ and $j_0$ are respectively a Bessel function and a spherical Bessel function of the first kind.

We know from~\eqref{imag-bessel} that $\Im \Phi(z,x_j)$ peaks at $z \approx x_j$ and decays as $z$ is away from $x_j$. Thus we can
determine $N$ as the number of significant peaks of  $|I(z)|^p$  where $p > 0$ is chosen to sharpen the significant peaks of the imaging functional (e.g. $p = 4$ works well
for the numerical examples in the last section). Then $x_j$  can be determined as the locations where $|I(z)|^p$ has significant peaks.   It is obvious that the resolution for imaging the sources is within the diffraction limit. 
Furthermore, from the fact $$J_0 (t)= \mathcal{O}(t^{-1/2}) \quad \text{ and } \quad j_0 (t)= \mathcal{O}(t^{-1}) $$ 
as $t \to \infty$,
we have
 $$|I(z)|^p = \mathcal{O} \Big ( \text{dist}(z,{ X})^{\frac{(1-d)p}{2}} \Big ) \quad \text{as} \quad \text{dist}(z,{ X}) \rightarrow \infty$$ 
where the set ${ X} = \{x_j : 1, \hdots, N \}$. Moreover, once $x_j$ is obtained the  intensity $\alpha_j$ can be estimated as
\begin{align}
\label{alpha}
\alpha_j \approx \frac{I (x_j)}{\Im \Phi(x_j,x_j)} \quad \text{since the $x_j$'s are well separated.} 
\end{align}
It is easy to verify that the imaging functional is stable with respect to noise in the Cauchy data see for instance \cite{Zhang2019}. We leave the proof to the readers. 

\section{Reconstruction of small volume sources}
We let $D$ with Lipschitz boundary $\partial D$  have small volume in $\R^d$ such that $|D| = \mathcal{O}(\epsilon^d)$. To this end, we assume that the scatterer $D$ is given by a collection of small volume subregions i.e.
\begin{equation}\label{sball}
D = \bigcup_{j=1}^{N} D_j \quad \text{with} \quad D_j = (x_j + \epsilon B_j) \quad \text{such that} \quad \text{dist}(x_i , x_j) \geq c_0>0
\end{equation}
for $ i \neq j$ where the parameter $0< \epsilon \ll 1$ and $B_j$ is a radially symmetric domain with Lipschitz boundary centered at the origin such that $|B_j|$ is independent of $\epsilon$. We also assume that the  $D_j$ are pairwise disjoint such that $D_j \cap D_i$ is empty for $ i \neq j$. Let $f \in L^2(D)$ where $\mathbbm{1}_{D}$ denotes the indicator function on the set $D$.
We assume that the radiating scattered field  $u \in H^{1}_{loc}(\R^d)$ generated by these small sources satisfies   
\begin{equation}\label{small-scat}
\Delta u + k^2 u = - f \mathbbm{1}_{D} \quad  \text{in} \quad \R^d 
\end{equation}
along with the Sommerfeld radiation condition as in~\eqref{radiation}.  
%We assume that the scattered field $u^s$ is radiating to close the system. Therefore, we impose the Sommerfeld radiation condition on the scattered field $u^s$ given by 
%\begin{equation}\label{SRC}
%{\partial_r u} - \text{i} ku =\mathcal{O} \left( \frac{1}{ r^{(d+1)/2} }\right) \quad \text{ as } \quad r \rightarrow \infty.
%\end{equation}
%The radiation condition is assumed to holds uniformly with respect to the angular variable $\hat{x}=x/r$ where $r=|x|$. Here, $| \cdot |$ denotes the Euclidean norm for a vector in $\R^d$. 
It is well--known that the solution is given by 
\begin{equation}\label{us-num}
u (x) =  \int_{D} f(y) \Phi(x,y)\, \text{d}y
\end{equation} 
where $\Phi (x,y)$ is the Green function as in~\eqref{green}.

We consider inverse problem  of determining $x_j$, $j = 1, \dots, N$, from  $u$  and $ \partial u/ \partial \nu$ on  $\partial \Omega$.  We have that for any $x \in \R^d \setminus \overline{D}$ we can use a first order Taylor expansion of the Green function to approximate 
the radiated field $u$. Therefore, we  notice that $y \in D_j$ if and only if $y = x_j + \epsilon \omega$ for some $\omega \in B_j$ which gives that 
\begin{align*}
u (x) &=  \int_{D} f(y) \Phi(x,y)\, \text{d}y 
	   = \sum_{j=1}^{N} \int_{D_j} f (y)  \Phi (x, x_j + \epsilon \omega) \, \text{d}y \\
	   &= \sum_{j=1}^{N} \Big (\Phi (x, x_j) + \mathcal{O}(\epsilon)   \Big ) \int_{D_j} f (y) \, \text{d}y. 
\end{align*}
Then, letting $\overline{f}_{D_j}$ be the average value of $f$ on $D_j$ and using the fact that $|D_j| = \epsilon^d |B_j|$  we have that 
\begin{equation}
\label{asym-solu}
u (x) =  \epsilon^d \sum_{j=1}^{N} \overline{f}_{D_j} |B_j|  \Phi (x, x_j)  + \mathcal{O}(\epsilon^{d+1}). 
\end{equation} 
This gives that up to leading order $u$ acts like a radiating scattered field generated by point sources located at $x_j$ with intensity $\alpha_j = \epsilon^d  |B_j| \overline{f}_{D_j} $. Notice, that by the Cauchy-Schwartz inequality we have that is bounded with respect to $\epsilon$. As in~\eqref{identity}, we use~\eqref{asym-solu} and the Green formula for $\Im \Phi(x,z)$ to derive the following identity.

%Therefore, we can plug \eqref{asym-solu} into the reciprocity gap functional to obtain 
\begin{lemma}
The following identity holds for the radiating scattered field $u$ in~\eqref{usc}
\begin{align}
\label{identity}
 \int_{\partial \Omega} \left(\frac{\partial \Im \Phi(x,z)}{\partial \nu(x)} u(x) - \Im \Phi(x,z) \frac{\partial u(x)}{\partial \nu(x)} \right) \dd s(x)= 
 \epsilon^d \sum_{j=1}^N \overline{f}_{D_j}  |B_j| \Im \Phi(x_j , z)+ \mathcal{O}(\epsilon^{d+1}).
\end{align}
\end{lemma}
%\begin{align*}
%\int_{\partial \Omega} \text{Im} \Phi(x,z) \partial_{\nu(x)} u^s(x) &- u^s(x) \partial_{\nu(x)} \text{Im} \Phi(x,z) \, \text{d}s(x) 
%= \epsilon^d \sum_{j=1}^M |B_j| Avg(f_j) \text{Im} \Phi(x_j , z)+ \mathcal{O}(\epsilon^{d+1}).
%\end{align*}
As in the case of point sources this lemma allows us to use the imaging functional $I(z)$ in~\eqref{functional} to determine $x_j$.
In fact, the centers $x_j$ of the small sources can be determined as the locations where  $|I(z)|^p$ has significant peaks.
%\begin{equation}\label{imaging-func}
%W(z) = \left|  \int_{\partial \Omega} \text{Im} \Phi(x,z) \partial_{\nu(x)} u^s(x) - u^s(x) \partial_{\nu(x)} \text{Im} \Phi(x,z) \, \text{d}s(x) \right|^p
%\end{equation} 
%for some fixed $p>0$. 
And also, up to the leading order, if $ \overline{f}_{D_j}  \neq 0$ we have that for all $z \in \R^d \setminus \overline{D}$, 
$| I(z) |^p = \mathcal{O} \Big ( \text{dist}(z,{ X})^{\frac{(1-d)p}{2}} \Big )$  as $ \text{dist}(z,{X}) \rightarrow \infty$
where the set ${ X} = \{x_j : 1, \hdots, N \}$.

\section{Numerical examples}

We present some  numerical examples in two dimensions to demonstrate the performance of the method.  
In all of the numerical examples the wave number is chosen as $k = 20$ and the Cauchy data is given at 256 points 
which are  uniformly distributed on
the circle centered at the origin with a radius of 50 (which is about 159 wavelengths). The data is computed using equations \eqref{usc} and \eqref{us-num}, respectively with 10\% random noise added to the data. The noise vectors  $\mathcal{N}_{1,2}$  consist of numbers $a + \text{i}b$ where $a, b \in (-1,1)$ are randomly generated with a uniform distribution. 
For simplicity we consider the same noise level for both $u$ and $\partial u/ \partial \nu$.  The noise vectors are added to  $u$ and $\partial u/ \partial \nu$ as follows
\begin{align*}
 u    + 10\%\frac{\mathcal{N}_1}{\|\mathcal{N}_1\|_F} \|u  \|_F, \quad 
  \frac{\partial u }{\partial \nu}  
   + 10\%\frac{\mathcal{N}_2}{\|\mathcal{N}_2\|_F} \left\|\frac{\partial u }{\partial \nu} \right \|_F,
\end{align*}
where $\|\cdot\|_F$ is the Frobenius matrix norm. 
The imaging functional $I(z)$ is computed at $256\times 256$ uniformly distributed sampling points on the search domain 
$[-2,2]^2$. That means the  distance between two consecutive sampling points is about 0.05 wavelength. This distance is small enough since the sources are supposed to be well separated. 

The  reconstruction results for  the location and intensity  of different numbers of point sources can be seen in Tables~\ref{tab1} and \ref{tab2} as well as  Figure~\ref{fi1}. The number of significant peaks of $|I(z)|^4$ (i.e. $p=4$) can be seen clearly in Figure~\ref{fi1}.  Then we determine $x_j$ as the location of these peaks. Tables~\ref{tab1} and \ref{tab2}  show that the locations
$x_j$ of the sources $\delta_{x_j}$ ($j = 1,\dots, N$) are computed with good accuracy and the intensities $\alpha_j$ are recovered with reasonable errors (the relative error ranges 
from 1\% to 10\%). In particular, the method can reasonably compute the location and intensity of two  point sources located within a wavelength.  It is noticed that 
the error gets worse as the number of sources to compute increases. 
We display in  Table~\ref{tab3} and Figure~\ref{fi2} the  reconstruction results of sources that are  small disks. The centers $x_j$ of these circular sources are again  computed 
with very good accuracy.

\begin{center}
\begin{table}[h]\centering
\begin{small}
\caption{Reconstruction result for the location of $N$ point sources}
\label{tab1}
\begin{tabular}{|c|c|c|c|cl}
\hline
$N$ & True location $x_j$ & Computed  location $x_j$   \\ \hline
2 (close) & $(0.15,0), \, (-0.15,0)$  & $(0.125,0.000), \,(-0.125,0.000)$  \\ \hline
2 (distant) & $(-1,0.8), \, (0.7,-1)$  & $(-1.000,0.796), \,(0.703,-1.000)$  \\ \hline
3 & $(1,-1),\, (1.3,1),\, (-1.2,-0.25)$ &  $(1.000,-0.984), \,(-1.312,0.984), \,(-1.203,-0.250)$   \\ \hline
4 & $(1,-1), (1,0.75),  (-1.2,-1), (-1.2,0.75)$ & $(1.015,-1.000),  (1.015,0.750),  (-1.203,-1.000),  (-1.203,0.750)$  \\
 \hline
\end{tabular}
\end{small}
\end{table}
\end{center}

\vspace{-1.2cm}

\begin{center}
\begin{table}[h]\centering
\begin{small}
\caption{Reconstruction result for the intensity of $N$ point sources with location in Table~\ref{tab1}}
\label{tab2}
\begin{tabular}{|c|c|c|c|cl}
\hline
$N$ & True intensity $\alpha_j$  & Computed intensity $\alpha_j$ \\ \hline
2 (close) &  {$1-2\text{i}, \, 1+2\text{i}$ }& { $ 0.942 - 1.881\text{i}, \, 0.950 + 1.899\text{i} $} \\ \hline
2 (distant) &  {$1-2\text{i}, \, 1+2\text{i}$ }& { $1.007-2.045\text{i}, \, 1.024+2.004\text{i} $} \\ \hline
3 &  {$4 + 0\text{i}, \, 3.5-\text{i},\, 3.5+\text{i}$ }   & { $ 3.912 + 0.023\text{i}, \, 3.431 - 0.918\text{i}, \, 3.570 + 0.968\text{i} $} \\ \hline
4 &  { $2.5+2\text{i}, \, 2.5 -2\text{i}, \, 3.5 - \text{i}, \, 3 + \text{i}$ }& { $2.725 + 2.237\text{i}, \, 2.759 - 2.158\text{i}, \, 3.532 - 1.035\text{i}, \,  2.982 + 1.077\text{i}$} \\
 \hline
\end{tabular}
\end{small}
\end{table}
\end{center}

\vspace{-1cm}

\begin{figure}[ht!!!]
\centering
\subfloat{\includegraphics[width=6cm]{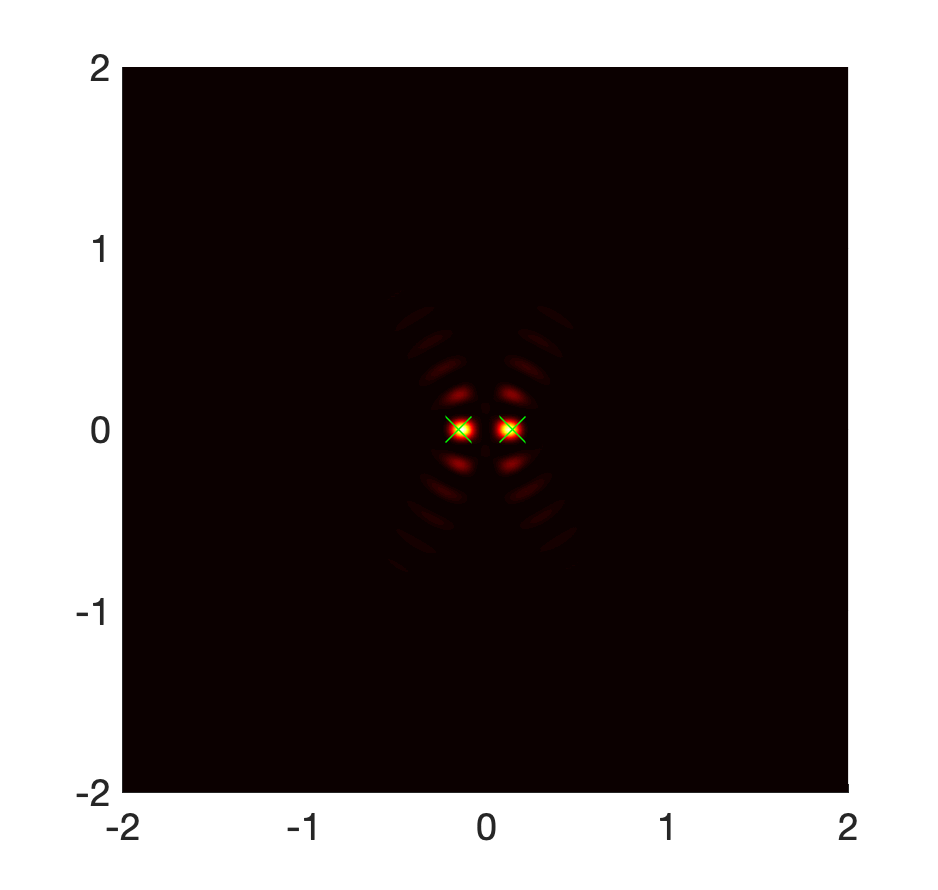}} \hspace{-0.5cm}
\subfloat{\includegraphics[width=6cm]{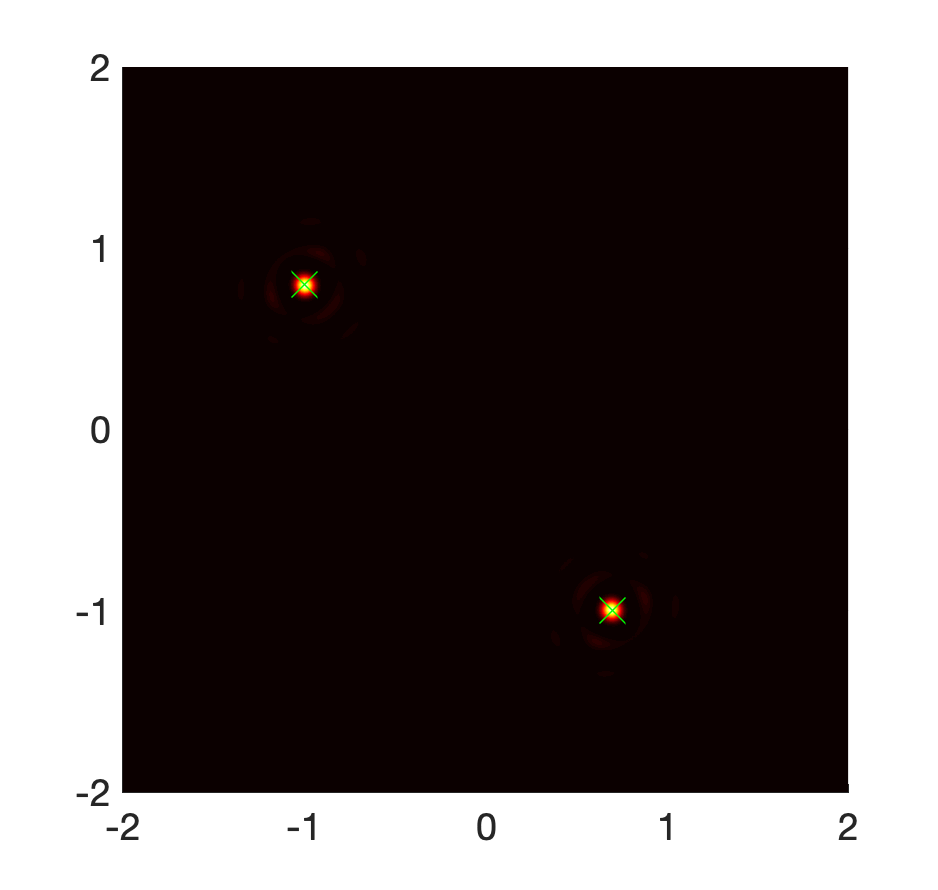}} \hspace{-0.5cm}\\
\subfloat{\includegraphics[width=6cm]{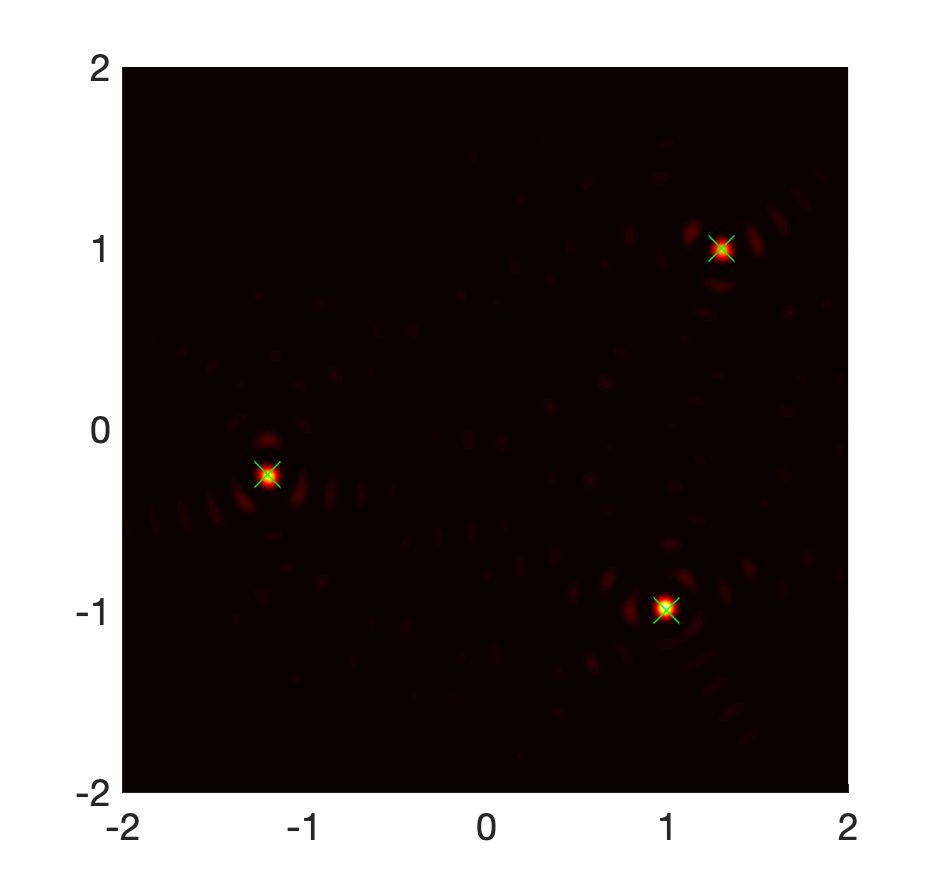}} \hspace{-0.5cm}
\subfloat{\includegraphics[width=6cm]{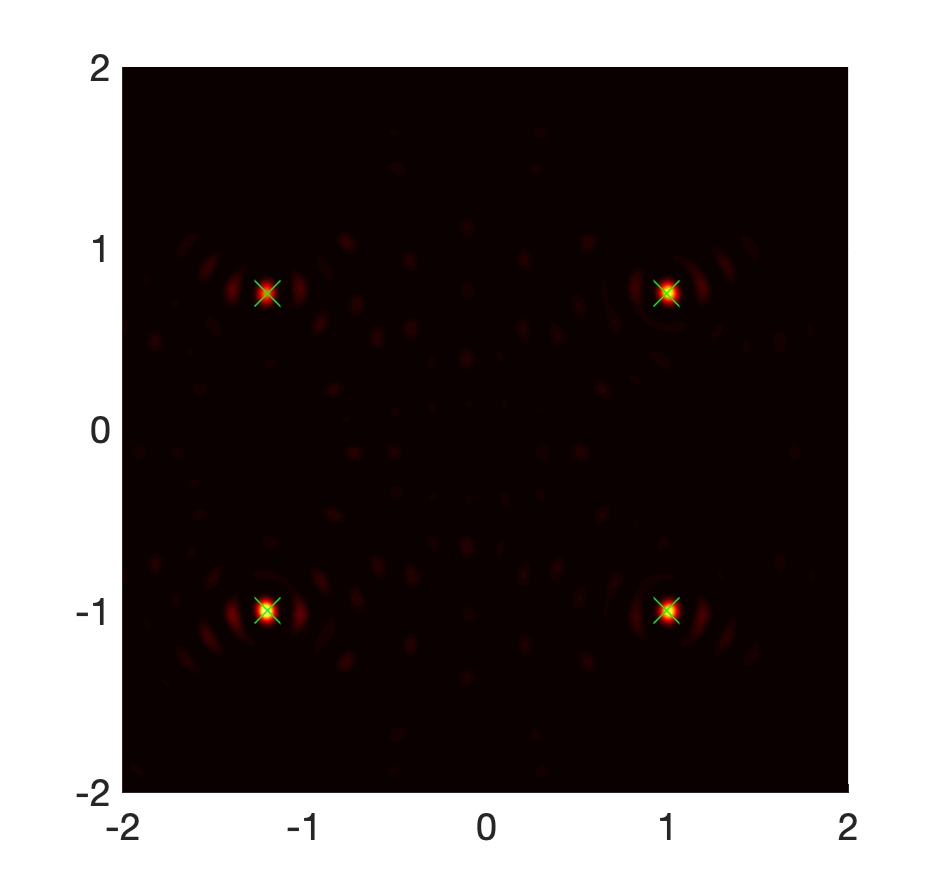}} \hspace{-0.5cm}
\caption{Pictures of $|I(z)|^4$ for point sources. The true location of the sources are displayed by the green crosses. 
 } 
 \label{fi1}
\end{figure}

\begin{table}[h]
\caption{Reconstruction result for the center $x_j$ of  $N$ small disks}
\label{tab3}\centering
\begin{footnotesize}
\begin{tabular}{|c|c|c|c|cl}
\hline
$N$ & True $x_j$ & Computed $x_j$ \\ \hline
3 & $(1,0.75), \,(-1,-1), \,(1.25,-1.5) $  & $(0.984,0.734),\, (-0.984,-1.000), \,(1.234,-1.500)$  \\ \hline
4 & $(1,1),\, (-1,-1.25), \,(1,-1), \,(-1,0.75)$ & $(0.984,0.984),\, (-0.984,-1.234), \,(0.984,-0.984), \,(-0.984,0.734)$  \\ \hline
5 & $(1.25,1.2), (-1,0), (1,-1), (-0.6,1),  (0.25,0) $ & $(1.265,1.203), (-1.000,0.015), (1.000,-1.015), (-0.609,0.984), (0.250,0.000)$  \\ \hline
\end{tabular}
\end{footnotesize}
\end{table}

\begin{center}
\begin{figure}[ht!!!]
\centering
\subfloat{\includegraphics[width=6cm]{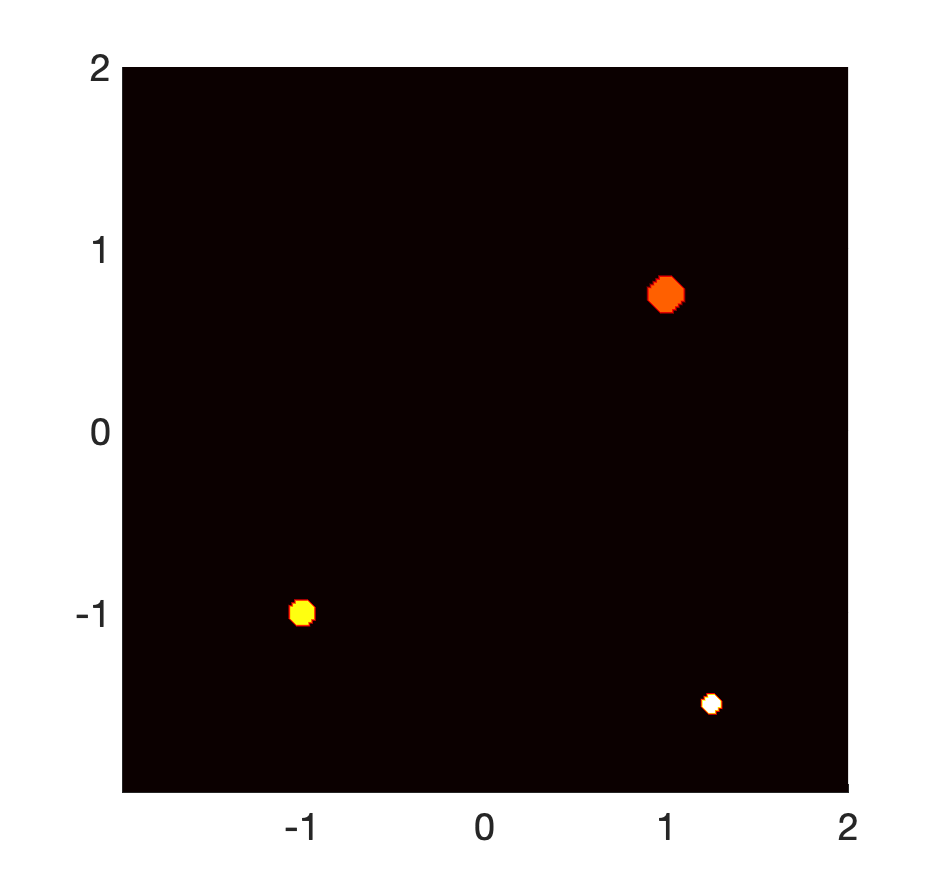}} \hspace{-0.5cm}
\subfloat{\includegraphics[width=6cm]{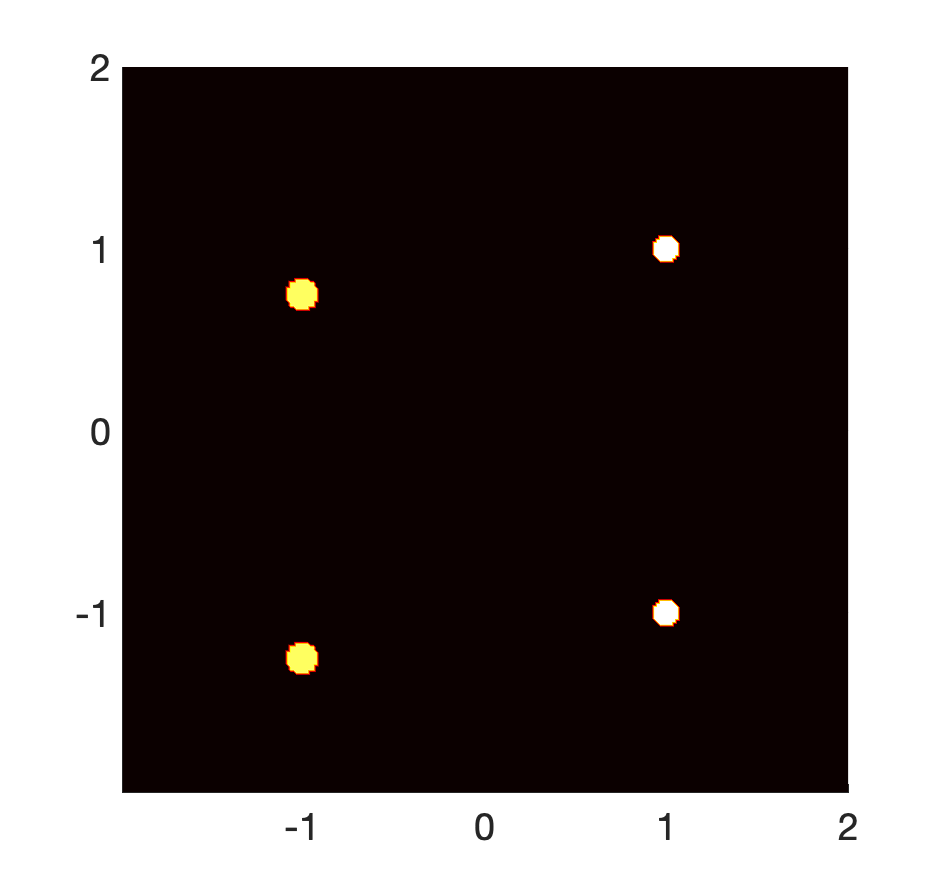}} \hspace{-0.5cm}
\subfloat{\includegraphics[width=6cm]{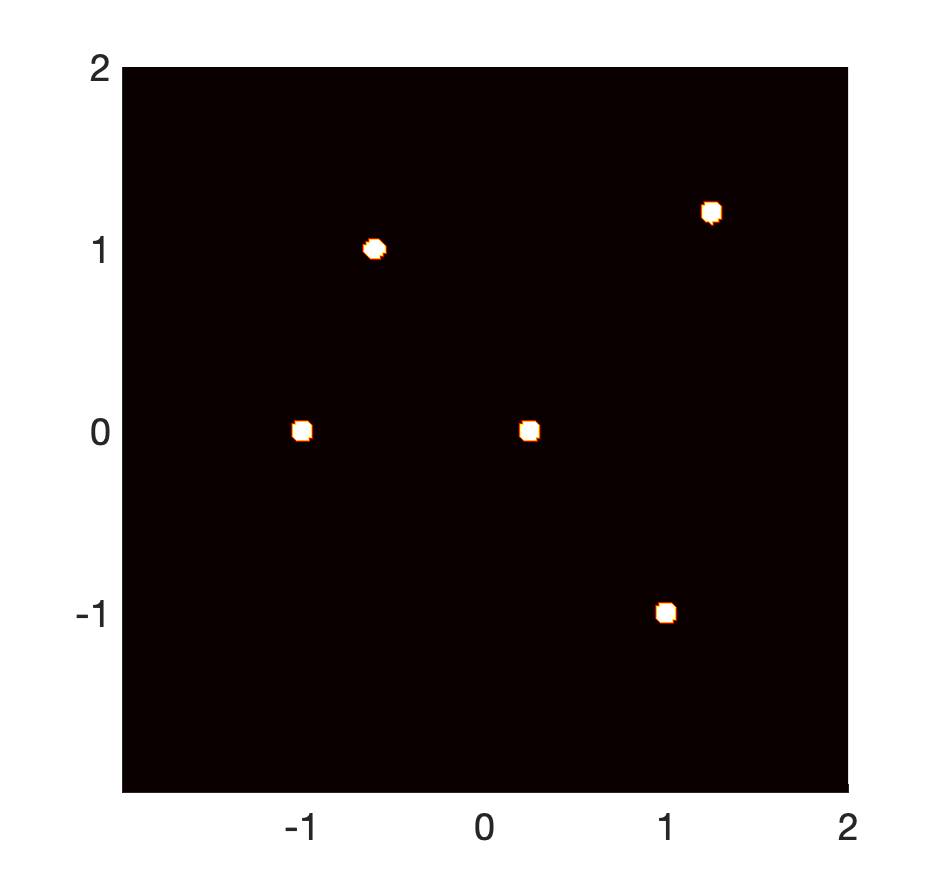}} \hspace{-0.5cm}\\
\noindent
\subfloat{\includegraphics[width=6cm]{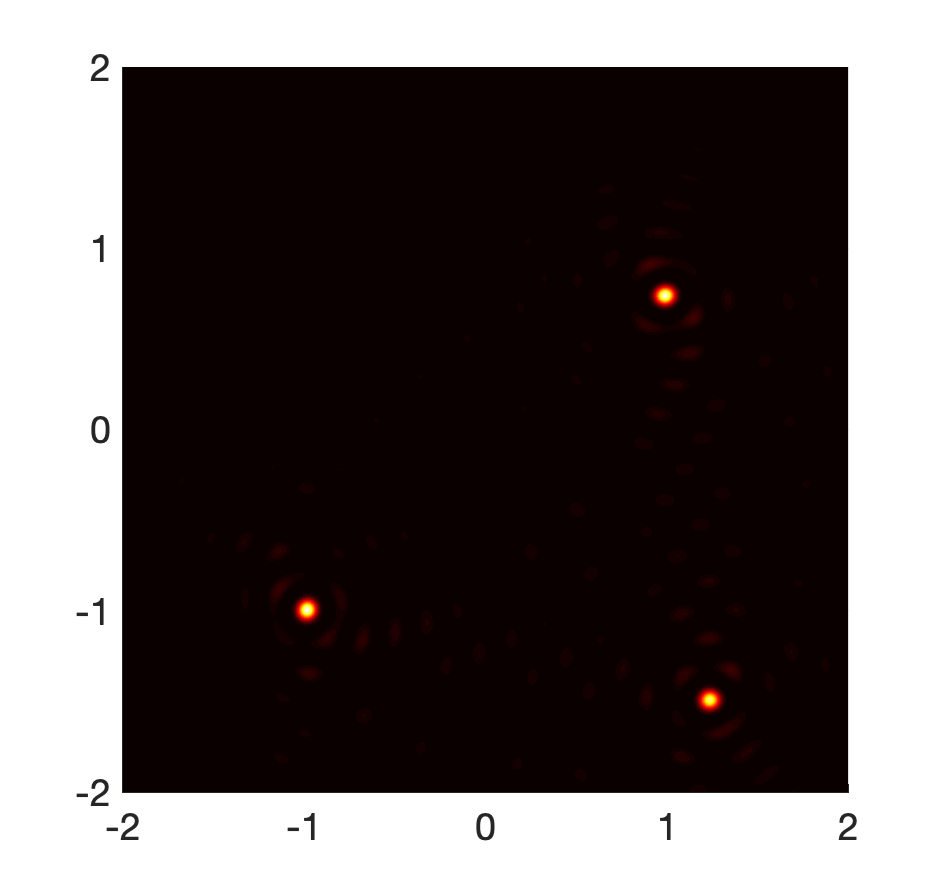}} \hspace{-0.5cm}
\subfloat{\includegraphics[width=6cm]{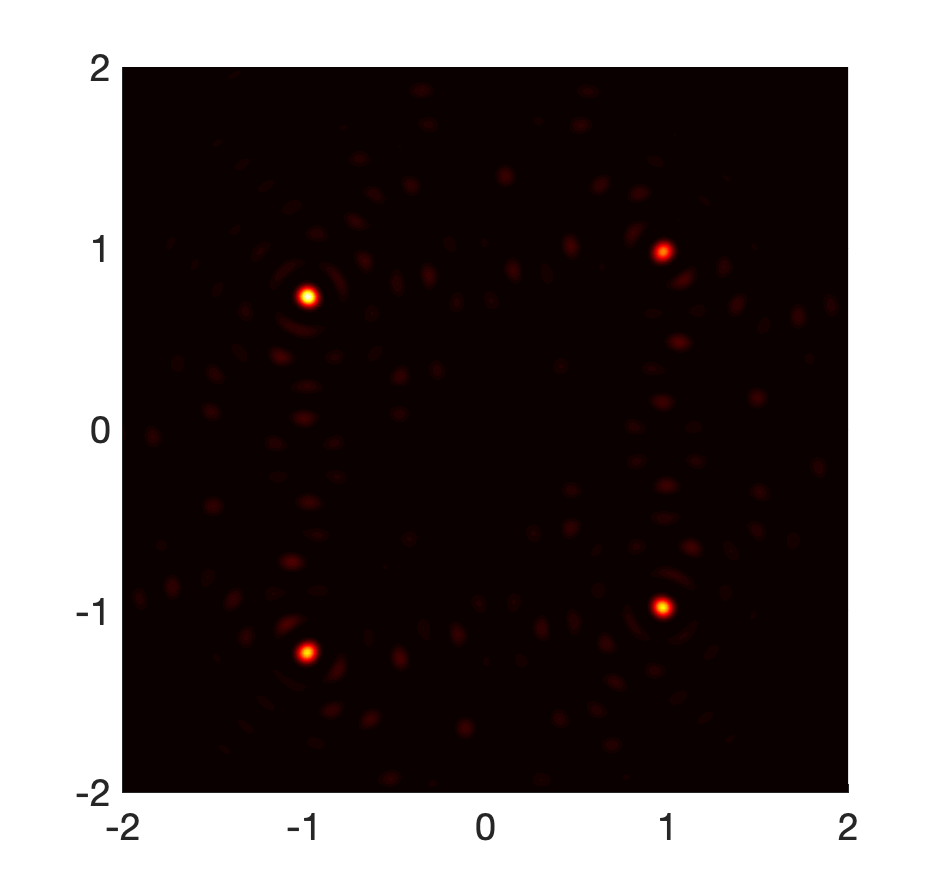}} \hspace{-0.5cm}
\subfloat{\includegraphics[width=6cm]{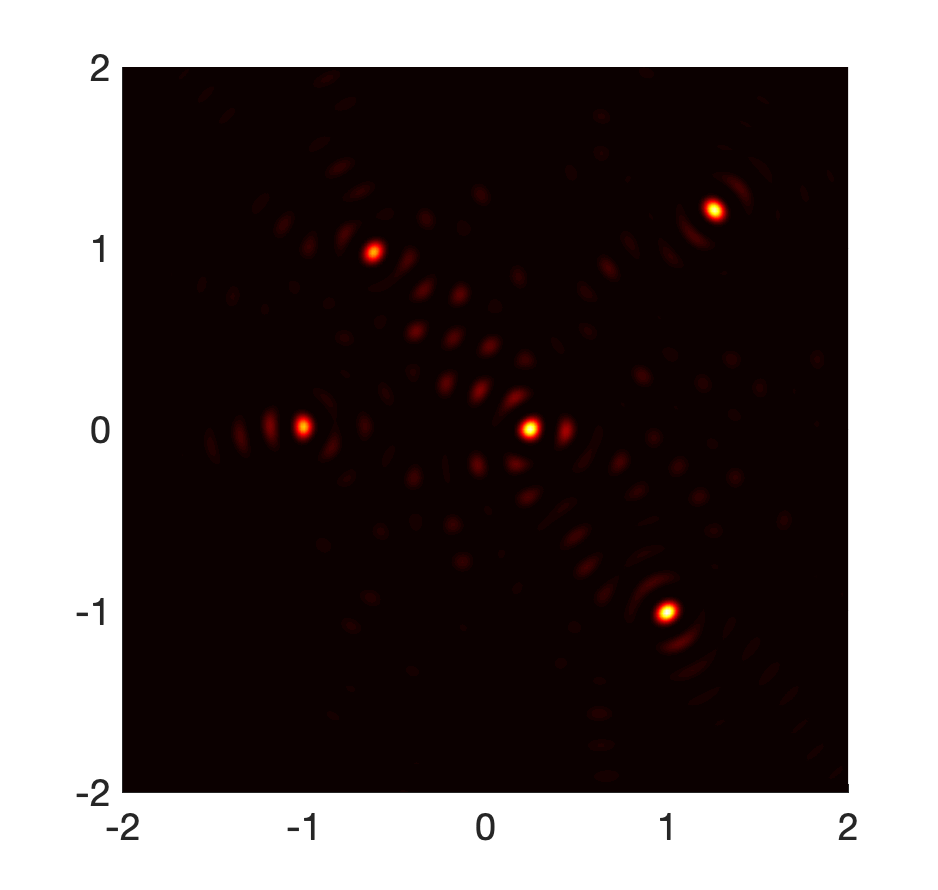}} \hspace{-0.5cm}
\caption{The first row is the true profile of small volume sources (the color of the disks is associated with the value of $f$ on each disk  that we are not able to recover using $I(z)$). The second row contains  pictures  of $|I(z)|^4$.
 } 
 \label{fi2}
\end{figure}
\end{center}
\noindent{\bf Acknowledgments:} The work of I. Harris was partially supported by the NSF  Grant DMS-2107891. 
The work of T. Le and D.-L. Nguyen was partially supported by NSF Grant DMS-2208293.

\bibliographystyle{plain}
\bibliography{ip-biblio2}

\end{document}

%% file: mathmacros.tex
\newcommand{\R}{\mathbb{R}}
\newcommand{\C}{\mathbb{C}}

\renewcommand{\phi}{\varphi}
\newcommand{\dd}{\mathrm{d}}

\renewcommand{\i}{\mathrm{i}}

\renewcommand{\Im}{\mathrm{Im}\,}